\documentclass[12pt,leqno]{amsart}
\newtheorem{thm}{Theorem}[section]
\newtheorem{defi}{Definition}[section]

\newcommand{\be}{\begin{equation}}
\newcommand{\ee}{\end{equation}}
\numberwithin{equation}{section}
\newcommand{\bea}{\begin{eqnarray}}
\newcommand{\eea}{\end{eqnarray}}
\newcommand{\beb}{\begin{eqnarray*}}
\newcommand{\eeb}{\end{eqnarray*}}
\usepackage{amssymb,amsfonts,amsthm,setspace,indentfirst}
\usepackage[dvips]{graphics}
\usepackage{epsfig}
\begin{document}
%%%%%%%%%%%%%%%%%%%%%%%%%%%%%%%%%%%%%%%%%%
\title{ A note on convergence of double sequences in a topological space}
\author{Amar Kumar Banerjee$^{1}$ and Rahul Mondal$^{2}$}
\address{\noindent\newline Department of Mathematics,\newline The University of 
Burdwan, \newline Golapbag, Burdwan-713104,\newline West Bengal, India.}   
\email{akbanerjee@math.buruniv.ac.in, akbanerjee1971@gmail.com}
\email{imondalrahul@gmail.com}
%%%%%%%%%%%%%%%%%%%%%%%%%%%%%%%%%%%%%%%%%%%%%%%%%%%%%%

\begin{abstract}
In this paper we have shown that a double sequence in a topological space satisfies certain conditions which in turn are capable to  generate a topology on a non empty set. Also we have used the idea of $I$-convergence of double sequences to study the idea of $I$-sequentially compactness \cite{AB} in the sense of double sequences.

\end{abstract}
%%%%%%%%%%%%%%%%%%%%%%%%%%%%%%%%%%%%%%%%%
\noindent\footnotetext{$\mathbf{2010}$\hspace{5pt}AMS\; Subject\; Classification: 54A20, 40A35, 40A05.}
\maketitle

\vspace{0.5in}

\section{\bf{Introduction}}
%%%%%%%%%%%%%%%%%%%%%%%%%%%%%%%%%%%%%%%
\noindent The idea of convergence topology on a non empty set i.e., the limit space arises from the properties of convergent sequence in a topological space. Indeed a sequence in a topological space satisfies the following conditions $L(1)$ to $L(3)$ in respect of convergence which in turn are capable to generate a topology on a non empty set which is called convergence topology and the topological space thus obtained is called limit space.\\
$L(1)$: For every element $p \in X$, the sequence $x= \left\{x_{i}\right\}$ where $x_{i} =p$ for all $i \in \mathbb{N} $, converges to $p$.\\
$L(2)$: Addition of finite numbers of terms to a convergent sequence affects neither its convergence nor the limit or limits to which it converges.\\
$L(3)$: If a sequence $x= \left\{x_{i}\right\}$ converges to a limit $p$, then every subsequence of it also converges to the same limit $p$.\\
 Thus in a topological space $(X, \tau)$, as the family of convergent sequences satisfies the conditions $L(1)$ to $L(3)$, it generates another topology $\sigma$ on X. Infact this topology $\sigma$ if finer then $\tau$.\\ 
\indent Here analogously we see that a double sequence in a topological space also satisfies such three types of similar conditions in respect of Pringsheim convergence \cite{PR} of double sequences. But these three conditions do not enable to yield such type of convergence topology on a non empty set. However a double sequence also satisfies two other conditions in a modified form of which the first one together with the conditions  $L(1)$ and $L(2)$ is sufficient to generate the convergence topology and second condition is necessary to correlate the convergence topology with the original topology.\\

The notion of statistical convergence of sequences was given by H. Fast \cite{FH}  and I. J. Schoenberg  \cite{SIJ} as an extension of the notion of convergence a real sequence. Later the idea of statistical convergence was extended to the notion of $I$-convergence of real sequences $($\cite{KP}, \cite{KPMM}    $)$ using the idea of ideal $I$ of subsets of the set of natural numbers. Many works were done in   $($\cite{BV},\cite{DK},\cite{KP},\cite{KPMM},\cite{LD1},\cite{MM} $)$ on $I$-convergence and statistical convergence.\\
For a double sequence the idea of statistical convergence was introduced by Muresaleen and Edely \cite{MU} and by M\'oricz \cite{MO} who introduced the notion statistical convergence for multiple sequences. In \cite{PD1} the idea of $I$-convergence was introduced for double sequences. Many works have been done on $I$-convergence of double sequences in $($\cite{BA}, \cite{DE}, \cite{TR}$)$.\\
\indent Recently some important properties on $I$-convergence and $I^{*}$-convergence of sequences and nets have been studied in a topological space and the idea of $I$-sequentially compactness in a topological space has been given in \cite{AB}. Here we also investigate some properties on the $I$-convergence of double sequences in a topological space and analogously study the idea of $I$-sequentially compactness in the sense of double sequences.

%%%%%%%%%%%%%%%%%%%%%%%%%%%%%%%%%%%%%%%%
\section{\bf{Preliminaries}}\label{preli}
%%%%%%%%%%%%%%%%%%%%%%%%%%%%%%%%%%%%%%%
%%%%%%%%%%%%%%%%%%%%%%%%%%%%%%%%%%%%%
\begin{defi} $($\cite{PR}$)$
A double sequence $\left\{x_{ij}\right\}$ in a topological space $(X, \tau)$ is said to converge to a point $\xi \in X$ in the Pringsheim's sense if for every open set U containing $\xi$, there exists a $k \in \mathbb{N}$ such that $x_{ij} \in U$ for all $i>k$ and $j>k$.
\end{defi}
\noindent The element $\xi$ is called Pringsheim limit of the double sequence $\left\{x_{ij}\right\}$ and is denoted by $P-lim_{i,j\rightarrow\infty} x_{ij} =\xi$.\\ 
\indent Throughout such type of convergence of double sequences will be called P-convergence or simply convergence. The Pringsheim's limit of a double sequence will be called P-limit or simply limit of the double sequence.
%%%%%%%%%%%%%%%%%%%%%%%%%%%%%%%%%%%%%%%%%%%%%%%%%%%%%%%
\begin{defi} 
$($\cite{KK}$)$ If $\textit{X}$ is a non-void set then a family of sets $\textit{I}\subset 2^{\textit{X}}$ is called an \textit{ideal} if \newline
$($i$)$ $A,B \in \textit{I}$ implies $A \cup B \in \textit{I}$ and \newline
$($ii$)$ $A\in \textit{I},B\subset A$ imply $B\in \textit{I}.$ 
\end{defi}
The ideal is called \textit{nontrivial} if $\textit{I} \neq \left\{\emptyset\right\}$ and $\textit{X}\notin \textit{I}$. A nontrivial ideal \textit{I} is called \textit{admissible} if it contains all the singleton sets. Several examples of nontrivial admissible ideals may be seen in \cite{KP}.\\ 
\indent Throughout X denotes a topological space, $\mathbb{R}$ for the set of all real numbers, $\mathbb{N}$ for the set of all natural numbers, $A^{c}$ denotes the complement of the set A and $I$ stand for a non trivial ideal of $\mathbb{N} \times \mathbb{N}$ unless otherwise stated.
%%%%%%%%%%%%%%%%%%%%%%%%%%%%%%%%%%%%%%%%
\begin{defi}$($\cite{KK}$)$ A nonempty family \textit{F} of subsets of a non-void set \textit{X} is called a \textit{filter} if \newline
$($i$)$  $\emptyset\notin \textit{F}$ \newline
$($ii$)$ $A,B\in \textit{F}$ implies $A\cap B\in \textit{F}$ and \newline
$($iii$)$ $A\in\textit{F},A\subset B$ imply $B\in \textit{F}.$ 
\end{defi}  
If \textit{I} is a nontrivial ideal on \textit{X} then $\textit{F}=\textit{F}(\textit{I})=\left\{A\subset \textit{X}:\textit{X}\setminus A \in \textit{I}\right\}$ is clearly a filter on \textit{X} and conversely.
%%%%%%%%%%%%%%%%%%%%%%%%%%%%%%%%%%%%%%%
\begin{defi}$($\cite{PD1}$)$
A non trivial ideal $I$ on $\mathbb{N} \times \mathbb{N}$ is called strongly admissible if $\left\{i\right\} \times \mathbb{N}$ and $\mathbb{N}\times\left\{i\right\}$ belong to $I$ for each $i \in \mathbb{N}$.\end{defi}
It should be noted that a strongly admissible ideal is also admissible.
%%%%%%%%%%%%%%%%%%%%%%%%%%%%%%%%%%%%%%%
\begin{defi}$($\cite{PD1}$)$
A double sequence $x=  \left\{x_{ij} \right\}$ in a topological space $(X, \tau)$ is said to be $I$-convergent to $x_{0}\in X$ if for any nonempty open set $U$ containing $x_{0}$, $\left\{(m,n)\in \mathbb{N} \times  \mathbb{N}:x_{mn}\notin U\right\}\in I$.
\end{defi}
%%%%%%%%%%%%%%%%%%%%%%%%%%%%%%%%%%%%%%
\noindent In this case we say that $x_{0}$ is the $I$-limit of $x$ and we write $I-lim x_{j,k} =x_{0}$. If $I$ is strongly admissible, then clearly $P$-convergence of $x$ implies $I$-convergence of $x$. But it should be noted that converse may not hold.\\
%%%%%%%%%%%%%%%%%%%%%%%%%%%%%%%%%%%%%%%
\indent In \cite{PD1} it is seen that there are ideals of $\mathbb{N} \times \mathbb{N}$ for which $I$-convergence coincides with Pringsheim convergence.

%%%%%%%%%%%%%%%%%%%%%%%%%%%%%%%%%%%%%%%
\begin{defi}$($\cite{PD2}$)$
Let $x=\left\{x_{ij}\right\}$ be a double sequence in a topological space $(X,\tau)$. Then $y\in X$ is called an $I$-cluster point of $x$ if for every open set $U$ containing $y$, $\left\{(m,n)\in \mathbb{N} \times \mathbb{N}:x_{mn}\in U\right\} \notin I$.
\end{defi}
%%%%%%%%%%%%%%%%%%%%%%%%%%%%%%%%%%%%
%%%%%%%%%%%%%%%%%%%%%%%%%%%%%%%%%%%%
%%%%%%%%%%%%%%%%%%%%%%%%%%%%%%%%%%%%%%%%%%%%%%%%%%%%%%%%%%%%%%%%%%%%%%%%%%%%%%%%%%%%%%%%%%%%%%%%%%%%%%%%%%%%%%%%%%%%%%%%%%%%%%%%%%%
\section{\bf{$d$-limit space}}
%%%%%%%%%%%%%%%%%%%%%%%%%%%%%%%%%%%%%%%%%%%%%%%%%%%%%%%
\indent The addition of terms in a double sequence is not simple as in the case of a single sequence. We now explain the method of addition of terms in a double sequence which can be done in two ways as follows.\\
\noindent$($a$)$ By shifting the elements of a particular row:\\
Let $\left\{x_{ij}\right\}$ be a double sequence in a topological space $(X, \tau)$.
If for some fixed $m\in\mathbb{N}$ we insert an additional term $y$ between two terms $x_{mn}$ and $x_{m \overline{n+1}}$, for $n\in \mathbb{N}$, then we get a new double sequence $\left\{y_{ij}\right\}$ such that for $i=m$, $y_{ij}=x_{ij}$ when $j \leq n$, $y_{ij}=y$ when $j=n+1$ and $y_{ij}=x_{i \overline{j-1}}$ when $j > n+1$ and for $i \neq m$, $y_{ij}=x_{ij}$ for all $j \in \mathbb{N}$. \\
\noindent$($b$)$ By shifting the elements of a particular column:\\
Similarly for some fixed $n \in \mathbb{N}$, we can insert $y$ between $x_{mn}$ and $x_{\overline{m+1} n}$ for some $m \in  \mathbb{N}$. In this case if the new double sequence formed be $\left\{y_{ij}\right\}$, then for $j=n$, $y_{ij}=x_{ij}$ when $i \leq m$, $y_{ij}=y$ when $i=m+1$, $y_{ij}=x_{\overline{i-1}j}$ when $i>m+1$ and for $j \neq n$, $y_{ij}=x_{ij}$ for all $i \in \mathbb{N}$.\\

\indent When we add finite number of terms $y_{1}, y_{2}, \dots y_{r}$ to a double sequence $\left\{x_{ij}\right\}$, we take an arbitrary member from the above r-terms say $y_{i}$, $1 \leq i \leq r$ and we add $y_{i}$ to the double sequence $\left\{x_{ij}\right\}$ as in the above process and a new double sequence $\left\{y_{ij}\right\}$ thus obtained. Then we take another member $y_{j}$, $1 \leq j \leq r$ , $j \neq i$ and adding  $y_{j}$ to the double sequence $\left\{y_{ij}\right\}$ we will have a new double sequence $\left\{z_{ij}\right\}$$($say$)$. Then continue the above process r-times to get the final double sequence by addition of finite number of terms $y_{1}, y_{2}, \dots y_{r}$.\\

\noindent \textbf{Note 3.1.} It should be noted that if $\left\{x_{ij}\right\}$ is a double sequence in a topological space $(X, \tau)$ converging to a point $p \in X$, then the sequence $\left\{x_{i} \right\}$ defined by $x_{i} = x_{ii}$ for all $i \in \mathbb{N}$, converges to $p$.\\

%%%%%%%%%%%%%%%%%%%%%%%%%%%%%%%%%%%%%%%%%%%%%%%%%%%%%%%%%%%%%%%% 
\begin{thm}
In a topological space $(X, \tau)$, the following statements holds for a convergent double sequence.\\
$(i)$ For every element $p \in X$, the double sequence $x= \left\{x_{ij}\right\}$ where $x_{ij} =p$ for all $(i,j) \in \mathbb{N} \times\mathbb{N} $, converges to $p$.\\
$(ii)$ Addition of finite numbers of terms to a convergent double sequence affects neither its convergence nor the limit or limits to which it converges.\\
$(iii)$ If $\left\{x_{ij}\right\}$ is a convergent double sequence in $A \cup B$ which converges to $p$, where A and B are two non empty disjoint subsets in X, then there exists a double sequence $\left\{y_{ij}\right\}$ $($whose range is a subset of the range set of $\left\{x_{ij}\right\}$$)$ either in A or in B consisting of infinitely many terms of $\left\{x_{ij}\right\}$ converging to $p$.\\
$(iv)$ If $\left\{x_{ij}\right\}$ is a double sequence converging to a point $p$ and if a sequence $\left\{x_{n}\right\}$ be formed such that, for each $n \in \mathbb{N}$, $x_{n}$ equals to some $x_{ij}$ where $i>n$ and $j>n$. Then the double sequence  $\left\{y_{ij}\right\}$ where for each $i \in \mathbb{N}$, $y_{ij}=x_{i}$ for all $j \in \mathbb{N}$, converges to $p$.\\
\end{thm}

\begin{proof}
$(i)$ Any open set containing p, contains all the terms of the double sequence $x$ and hence it converges to $p$.\\
$(ii)$ Let $x= \left\{x_{ij}\right\}$ be a double sequence converging to $l$. Then for any open set U containing $l$, there exists a positive integer $m$ such that $x_{ij} \in U$ for all $i>m$ and $j>m$. Now let r-terms $y_{1}, y_{2}, \dots y_{r}$ be added to the double sequence $x$ and the new double sequence formed be $\left\{z_{ij}\right\}$. So there exists $(p,q) \in \mathbb{N}\times \mathbb{N}$ such that for each $s$, $1 \leq s \leq r$, $y_{s} =z_{i_{s}j_{s}}$ where $i_{s} \leq p$ and $j_{s} \leq q$ and $z_{ij} = x_{ij}$, for $i>p$ and $j>q$.
Now let $k$ be a positive integer greater then $m+p$ and $m+q$ then $z_{ij} \in U$ for all $i>k$ and $j>k$. Hence $\left\{z_{ij}\right\}$ is also converges to $l$.\\
$(iii)$ Let $\left\{x_{ij}\right\}$ be a double sequence in $A \cup B$ converging to $p$ where A and B are two nonempty disjoint sets in X. We consider the sequence $\left\{x_{i}\right\}$ where $x_{i}= x_{ii}$ for all $i \in \mathbb{N}$. Clearly $\left\{x_{i}\right\}$ is in $A \cup B$ which also converges to $p$.
Now at least one of A or B must contains a subsequence of $\left\{x_{i}\right\}$, which converges to $p$. Without any loss of generality we suppose that A contains a subsequence $\left\{x_{i_{k}}\right\}$ of $\left\{x_{i}\right\}$ converging to $p$. Now let us consider a double sequence $\left\{y_{kj}\right\}$ where for each $k \in \mathbb{N}$, $y_{kj} = x_{i_{k}}$ for all $j \in  \mathbb{N}$. Let U be any open set containing $p$. Then there exists a $m \in \mathbb{N}$ such that $x_{i_{k}} \in U$ for all $k>m$. Hence for all $k>m$ and $j>m$, $y_{kj} \in U$. Therefore $\left\{y_{kj}\right\}$ converges to $p$.\\
$(iv)$ Let U be an open set containing $p$. Then there exists a $k \in \mathbb{N}$ such that $x_{ij} \in U$ for all $i>k$ and $j>k$. Hence $x_{n} \in U$ for all $n>k$, since for $n>k$, $x_{n}=x_{ij}$ for some $i,j \in \mathbb{N}$ such that $i>n>k$ and $j>n>k$. Therefore $y_{ij} \in U$ for all $i>k$ and $j>k$. Hence the double sequence $\left\{y_{ij}\right\}$ converges to $p$.\\
\end{proof}
%%%%%%%%%%%%%%%%%%%%%%%%%%%%%%%%%%%%%%%%%%%%%%%%%%%%%%%%%%%%%%%%%%%%%%%%%%%%%%%%%%%%%%
\indent Let $x= \left\{x_{ij}\right\}$ be a double sequence and let $\left\{i_{p}\right\}$ and $\left\{j_{q}\right\}$ be two strictly increasing sequences of natural numbers. Then $\left\{x_{i_{p}j_{q}}\right\}$ is said to be a subsequence of $x$. It should be noted that if A and B are two non empty disjoint sets in a topological space and if $x= \left\{x_{ij}\right\}$ is a double sequence in $A \cup B$, then neither of A and B may contain a subsequence of $x$.
%%%%%%%%%%%%%%%%%%%%%%%%%%%%%%%%%%%%%%%%%%%%%%%%%%%%%%%%%%%%%%%%%%%%%%%%%%%%%%%%%%%%%%%
\begin{thm}
If G be an open set in a topological space $(X, \tau)$ then no double sequence lying in $X \setminus G$ has any limit in G. 
\end{thm}
\indent The proof is straightforward and so is omitted.

%%%%%%%%%%%%%%%%%%%%%%%%%%%%%%%%%%%%%%%%%%%%%%%%%%%%%%%%%%%%%%%%%%%%%%%%%%%%%%%%%%%%%%%
\begin{thm}
Let X be a given set and let $\Omega$ be a class of double sequences over X. Let the members of $\Omega$ be called convergent double sequences and let each convergent double sequence be associated with an element of X called limit of the convergent double sequence subject to the conditions $(i)$ to $(iv)$ as stated in theorem 3.1.\\
Now let a subset G of X be called open if and only if no convergent double sequence lying in $X \setminus G$ has any limit in G. Then the collection of open sets thus obtained forms a topology $\tau$ on X.
\end{thm}

\begin{proof}
Since $\phi$ contains no element and no convergent double sequence can lie outside X, it follows that X and $\phi$ are open sets.\\
Let $\left\{G_{\alpha} \right\}_{\alpha \in \land}$ be a collection of open sets, where $\land$ is an arbitrary indexing set. Let $G= \bigcup _{\alpha \in \land} G_{\alpha}$ and let us consider a convergent double sequence $x= \left\{x_{ij}\right\}$ in $G^{c}$. Then $x$ will be in $G_{\alpha}^{c}$ for all $\alpha \in \land$. Hence $x$ can not have a limit in $G_{\alpha}$ for all $\alpha \in \land$. Therefore $x$ can not have a limit in G. Thus G is open.\\
Now let G and H be two open sets and if possible let $x= \left\{x_{ij}\right\}$ be a convergent double sequence in $(G\cap H)^{c}$, which has a limit $p$ in $G\cap H$. So by the condition $($iii$)$ at least one of $(G\cup H)^{c}$ and $(G\cap H^{c}) \cup (G^{c}\cap H)$ must contain a convergent double sequence whose range set is a subset of range set of $x$ and which has a limit $p$. If $(G \cup H)^{c}$ contains a convergent double sequence $y= \left\{y_{ij}\right\}$ whose range set is a subset of the range set of $x = \left\{x_{ij}\right\}$ and which has a limit $p$ in $G \cap H \subset G \cup H$, then it will contradict the fact that $G \cup H$ is open. So suppose that $(G \cap H^{c})\cup(G^{c} \cap H)$ contains a convergent double sequence $y= \left\{y_{ij}\right\}$ whose range set is a subset of range set of $x = \left\{x_{ij}\right\}$ and which has a limit $p$. Then again by the condition $($iii$)$ at least one of $(G\cap H^{c})$ and $(G^{c}\cap H)$ must contains a convergent double sequence say $z= \left\{z_{ij}\right\}$ whose range set is a subset of range set of $y$ and which has a limit $p$.
Without any loss of generality assume that $(G\cap H^{c})$ contains the convergent double sequence $z= \left\{z_{ij}\right\}$. Then we have a convergent double sequence $z= \left\{z_{ij}\right\}$ in $H^{c}$ which has a limit $p$ in H, but this is again a contradiction to the fact that H is open. Hence no convergent double sequence in $(G\cap H)^{c}$ can have a limit in $(G\cap H)$. Therefore $(G\cap H)$ is open. Thus $\tau$ forms a topology on X.
\end{proof}
%%%%%%%%%%%%%%%%%%%%%%%%%%%%%%%%%%%%%%%%%%%%%%%%%%%%%%
\noindent  We call the topology defined above the convergence topology on X and $(X, \tau)$ is called the double limit space or in short $d$-limit space.
%%%%%%%%%%%%%%%%%%%%%%%%%%%%%%%%
\begin{thm}
Let X be a $d$-limit space with $\Omega$ as the given collection of convergent double sequences and $\tau$ be the resulting convergence topology on X. If $\Sigma$ is the collection of all convergent double sequences determined by the topology $\tau$ on X, then $\Omega \subset \Sigma$.
\end{thm}

\begin{proof}
Suppose $x= \left\{x_{ij}\right\}$ be a member of $\Omega$ and $l$ be a limit of $x$. Let U be an open set in the $d$-limit space $(X, \tau)$ containing $l$. We claim that there exists a $k \in \mathbb{N}$ such that $x_{ij} \in U$ for all $i>k$ and $j>k$. For otherwise, for every $m \in \mathbb{N}$, there exists a $x_{ij} \in U^{c}$ where $i>m$ and $j>m$. Then we can construct a sequence $\left\{x_{n}\right\}$ in $U^{c}$ where for each $n \in \mathbb{N}$, $x_{n}$ equals to some $x_{ij}$, where $i>n$ and $j>n$. Hence the double sequence $\left\{y_{ij}\right\}$ where for each $i \in \mathbb{N}$, $y_{ij}=x_{i}$ for all $j \in \mathbb{N}$, converges to $l$, by the condition $(iv)$. Therefore we have a double sequence $\left\{y_{ij}\right\}$ in $U^{c}$ which has a limit in U. This contradicts to the fact that U is open. Hence  $\left\{x_{ij}\right\}$ converges to $l$ with respect to $\tau$ and so  $x \in \Sigma$.
\end{proof}
%%%%%%%%%%%%%%%%%%%%%%%%%%%%%%%%%%%%%%%%%%%%%%%%%%%%%%%%%%%%%%%%%%%%%%%%
\begin{thm}
Let $\Gamma$ be the family of all convergent double sequences in a topological space $(X, \tau)$ and $\tau ^{'}$ be the convergence topology on X determined by the family $\Gamma$. Then $\tau \subset \tau ^{'}$.
\end{thm}
\begin{proof} Indeed each member of $\Gamma$ satisfies the conditions $(i)$ to $(iv)$, by theorem 3.1.
Let G be a $\tau$-open set. If possible let G be not $ \tau ^{'}$-open. Then there exists a double sequence $x= \left\{x_{ij}\right\}$ in $G^{c}$ which has a limit $p$ in G in the sense of theorem 3.3. Since G is $\tau$-open set and $p \in G$. So there exists a positive integer $m$ such that $x_{ij} \in G$ for all $i>m$ and $j>m$. This is a contradiction to the fact that $x= \left\{x_{ij}\right\}$ is in $G^{c}$. Hence G must be $ \tau ^{'}$-open.
\end{proof}
%%%%%%%%%%%%%%%%%%%%%%%%%%%%%%%%%%%%%%%%%%%%%%%%%%%%%%%%%%%%%%%%%%%%%%%%%
\section{\bf{$I$-convergence of double sequences}}
%%%%%%%%%%%%%%%%%%%%%%%%%%%%%%%%%%%%%%%%%%%%%%%%%
\noindent Recall that a point $p$ in a topological space  $(X, \tau)$  is called an $\omega$-accumulation point of a  subset A if every neighbourhood of $p$ intersects A in infinitely many points.
%%%%%%%%%%%%%%%%%%%%%%%%%%%%%%%%%%%%%%%%%%%%%%%%%
\begin{thm}
Let $I$ be a non trivial ideal of the set $\mathbb{N} \times \mathbb{N}$ consisting of all finite subsets of $\mathbb{N} \times \mathbb{N}$. If every infinite set in a topological space $(X, \tau)$ has an $\omega$-accumulation point, then every double sequence $\left\{x_{ij}\right\}$ has an $I$-cluster point.
\end{thm}

\begin{proof}
Let $\left\{x_{ij}\right\}$ be a double sequence in X. If the range set of $\left\{x_{ij}\right\}$ is infinite then let us denote it by M. Then M has an $\omega$-accumulation point $y$ $($say$)$. Now let V be an open set containing $y$. Then clearly V contains infinitely many members of M. Hence the set $\left\{ (i,j)\in \mathbb{N} \times \mathbb{N}: x_{ij} \in V\right\}$ is infinite. Therefore $\left\{ (i,j)\in \mathbb{N} \times \mathbb{N}: x_{ij}\in V\right\} \notin I$. Thus $y$ becomes $I$-cluster point of $\left\{x_{ij}\right\}$.\\
If the range set of $\left\{x_{ij}\right\}$ is finite then there exists a $y \in X$ such that $x_{ij} = y$ for infinitely many $(i,j)\in \mathbb{N} \times \mathbb{N}$. Then for every open set V containing $y$, $\left\{ (i,j)\in \mathbb{N} \times \mathbb{N}: x_{ij}\in V\right\} $ is an infinite set. Hence $\left\{ (i,j)\in \mathbb{N} \times \mathbb{N}: x_{ij}\in V\right\} \notin I$. Thus $y$ becomes an $I$-cluster point of $\left\{x_{ij}\right\}$.
\end{proof}
%%%%%%%%%%%%%%%%%%%%%%%%%%%%%%%%%%%%%%%%%%%%%%%%%%%%%%%%%%%%%%%%%%%%%
\noindent \textbf{Note 4.1.} The converse of the above result is also true in a topological space for a single sequence. But for a double sequence the converse holds if the ideal $I$ contains all sets of the form $H \times \mathbb{N}$, where H is a finite subset of $\mathbb{N}$.
%%%%%%%%%%%%%%%%%%%%%%%%%%%%%%%%%%%%%%%%%%%%%%%%%%%%%%%%%%%%%%%%%%%%%
\begin{thm}
Let $I$ be a non trivial ideal of the set $\mathbb{N} \times \mathbb{N}$ which contains the subsets of $\mathbb{N} \times \mathbb{N}$ of the form $H \times \mathbb{N}$, where H is a finite subset of $\mathbb{N}$. Then if every double sequence in a topological space $(X, \tau)$ has an $I$-cluster point, then every infinite subset of X has an $\omega$-accumulation point.
\end{thm}

\begin{proof}
Suppose that every double sequence in $(X, \tau)$ has an $I$-cluster point. Let A be an infinite subset of X. Then there exists a sequence $\left\{a_{i}\right\}_{i\in \mathbb{N}}$ of distinct points of A. Now let us define a double sequence $\left\{x_{ij}\right\}$ in A where for each $i \in \mathbb{N}$, $x_{ij}= a_{i}$ for all $j \in \mathbb{N}$. Let $y$ be an $I$-cluster point of $\left\{x_{ij}\right\}$. Then for any open set V containing $y$, $\left\{ (i,j)\in \mathbb{N} \times \mathbb{N}: x_{ij}\in V\right\} \notin I$. Hence V contains infinitely many members of A. Hence $y$ becomes $\omega$-accumulation point of A.
\end{proof}
%%%%%%%%%%%%%%%%%%%%%%%%%%%%%%%%%%%%%%%%%%%%%%%%%%%%%%%%%%%%%%%%%%%%%%%%%%
\begin{thm}
Let $I$ be an non trivial ideal of $\mathbb{N} \times \mathbb{N}$ containing all the sets of the form $H \times \mathbb{N}$, where H is a finite subset of $\mathbb{N}$. If $(X, \tau)$ is a Lindel\"of space such that every double sequence has an $I$-cluster point then $(X, \tau)$ is compact.
\end{thm}

\begin{proof}
Let $\mathcal{A}=\left\{A_{\alpha} :\alpha \in \Lambda\right\}$ be an open cover for X, where $\Lambda$ is an arbitrary indexing set. Since X is Lindel\"of, there exists a countable sub cover $\left\{ A_{1},A_{2}, \dots A_{n} \dots\right\}(say)$ of $\mathcal{A}$ for X. Let $B_{1}=A_{1}$ and for each $m>1$ let $B_{m}$ be the first member of $A_{i}$'s which is not covered by $B_{1} \cup B_{2}\cup \dots \cup B_{m- 1}$.
If this choice becomes impossible at any stage then the collection of selected sets becomes a finite sub cover for X. Otherwise we can construct a countable family $\left\{C_{i}\right\}_{i \in \mathbb{N}}$ of sets as follows: $C_{1}=B_{1}=A_{1}$, $C_{2}= B_{2} \setminus B_{1}$, $C_{3}= B_{3} \setminus B_{1} \cup B_{2}$, ..............................and so on.
Then clearly $C_{i} \cap C_{j} = \phi$ for $i \neq j$ and $C_{i} \neq \phi$ for all $i \in \mathbb{N}$. For each $i \in \mathbb{N}$ let us choose an arbitrary element $a_{i}$ in $C_{i}$ and construct a double sequence $\left\{x_{ij}\right\}$ where for each $i \in \mathbb{N}$, $x_{ij}= a_{i}$ for all $j \in \mathbb{N}$. Now let $y$ be an $I$-cluster point of $\left\{x_{ij}\right\}$ and let $y \in B_{p}$. So the set $M=\left\{(i,j) \in \mathbb{N} \times \mathbb{N} : x_{ij} \in B_{p}\right\} \notin I$. Since $I$ contains all the sets of the form $H \times \mathbb{N}$ where H is a finite subset of $\mathbb{N}$, M must be of the form $H \times K$ where H is an infinite subset of $\mathbb{N}$. So there exists $m>p$ such that $x_{mn} \in B_{p}$.
According to the construction, $x_{mn}=a_{m}$ must belong to any one of $C_{1},C_{2} \dots C_{p}$ but $x_{mn} \notin C_{m}$, since $m>p$. By the construction of the double sequence, $x_{mn}=a_{m}$ must belongs to $C_{m}$, which is a contradiction. Thus the result follows.
\end{proof}
%%%%%%%%%%%%%%%%%%%%%%%%%%%%%%%%%%%%%%%%%%%%%%%
\begin{defi}$($cf. \cite{AB}$)$
A topological space $(X,\tau)$ is said to be $I$-sequentially compact in the sense of double sequences if every double sequence in $X$ has an $I$-cluster point where $I$ is a non-trivial ideal of the set $\mathbb {N} \times \mathbb{N}$.
\end{defi}
%%%%%%%%%%%%%%%%%%%%%%%%%%%%%%%%%%%%%%%%%%%%%%
\begin{thm}
Let $I$ be a non trivial ideal of $\mathbb{N} \times \mathbb{N}$ containing all the sets of the form $H \times \mathbb{N}$, where H is a finite subset of $\mathbb{N}$. Then if $(X, \tau)$ is $I$-sequentially compact, $(X, \tau)$ becomes a countably compact space.
\end{thm}

\begin{proof}

Let $(X, \tau)$ be an $I$-sequentially compact space. If possible let $\left\{V_{n}\right\}_{n=1} ^{\infty}$ be a countable open cover of X which has no finite subcover. Hence $X \setminus \bigcup _{i=1} ^{n} V_{i} \neq \phi$ for all $n \in \mathbb{N}$. Now let us take an arbitrary element $a_{n} \in X \setminus \bigcup _{i=1} ^{n} V_{i}$ for each $n \in \mathbb{N}$ and consider the double sequence $\left\{x_{ij}\right\}$ where for each $n \in \mathbb{N}$, $x_{ij}=a_{i}$ for all $j \in \mathbb{N}$. Now the double sequence $\left\{x_{ij}\right\}$ has an $I$-cluster point $x_{0}(say)$ in X, since $(X,\tau)$ is $I$-sequentially compact. Let $x_{0} \in V_{r}$ for some $r \in \mathbb{N}$. Then from the definition of $I$-cluster point, the set $M= \left\{(i,j) \in \mathbb{N} \times \mathbb{N}: x_{ij} \in V_{r} \right\} \notin I$. Since $I$ contains all the sets of the form $H \times \mathbb{N}$ where H is a finite subset of $\mathbb{N}$, M must be of the form $K \times J$ where K is an infinite set. So there exists $(m,n) \in \mathbb{N} \times \mathbb{N}$ with $m>r$ such that $x_{mn}\in V_{r}$. But $x_{mn} \in X \setminus \bigcup _{i=1} ^{m} V_{i}$, that is, $x_{mn} \notin V_{r} $ which is a contradiction. So X must be countably compact.
\end{proof}
%%%%%%%%%%%%%%%%%%%%%%%%%%%%%%%%%%%%%%%%%%%%%%%%%%%%%%%%%%%%
\indent Note that if $I$ is a strongly admissible ideal of $\mathbb{N} \times \mathbb{N}$, and if a double sequence $\left\{x_{ij}\right\}$ in a topological space $(X, \tau)$ is convergent in the sense of Pringsheim to $x_{0}$, then $\left\{x_{ij}\right\}$ is $I$-convergent to $x_{0}$.\\
%%%%%%%%%%%%%%%%%%%%%%%%%%%%%%%%%%%%%%%%%%%%%%%%%%%%%%%%%%
Now we recall the following widely known result which will be needed in the next theorem.   \newline
\noindent \textbf{Lemma.} In a topological space $(X,\tau)$ the following are equivalent.   \\
(a) $(X,\tau)$ is countably compact.  \\
(b) Every countable collection of closed subsets of $X$ satisfying the finite intersection property has non-empty intersection.  \\
(c) If $F_{1}\supset F_{2}\supset F_{3}\supset \cdots \supset F_{n}\supset \cdots$ is a descending family of non-empty closed subsets of $X$, then $\bigcap_{n=1}^{\infty} F_{n} \neq \emptyset$. \\
%%%%%%%%%%%%%%%%%%%%%%%%%%%%%%%%%%%%%%%%%%%%%%%%%%%%%%%%%%%%%%
\begin{thm}
Let $I$ be a strongly admissible ideal of $\mathbb{N} \times \mathbb{N}$. If $(X,\tau)$ is a first countable countably compact space then $(X,\tau)$ becomes an $I$-sequentially compact space in the sense of double sequence. 
\end{thm}

\begin{proof}
Let $(X,\tau)$ be a first countable countably compact space and let $\left\{x_{ij}\right\}$ be a double sequence in X. We consider $P_{n}=\left\{x_{ij}: i\geq n , j \geq n\right\}$ for $n \in \mathbb{N}$. Let R be the range set of the double sequence $\left\{x_{ij}\right\}$ and let $T_{n}=\left\{x_{ij} \in R :x_{ij} \in P_{n}\right\}$.
Now $\left\{ \overline{T_{n}} \right\}$ is a descending family of nonempty closed sets and hence by above lemma we have $\bigcap _{n=1} ^{\infty} \overline{T_{n}} \neq \phi$. Let $y \in \bigcap _{n=1} ^{\infty} \overline{T_{n}}$. Since $(X,\tau)$ is first countable there exists a nested countable local base $\left\{B_{n}(y)\right\}_{n=1} ^{\infty}$ at $y$. Also we have $B_{n}(y) \cap T_{n} \neq \phi $ for each n. Now let us take an arbitrary element $a_{n} \in B_{n}(y) \cap T_{n}$, for each $n \in \mathbb{N}$. We define a double sequence $\left\{y_{ij} \right\}$ where for each $i \in \mathbb{N}$, $y_{ij} = a_{i}$ for all $j \in \mathbb{N}$. We shall show that $\left\{y_{ij}\right\}$ converges to $y$ in Pringsheim's sense. Let V be an arbitrary open set containing $y$. Then there exists a $k\in \mathbb{N}$ such that $y \in B_{k} \subset V$ and since $B_{n}(y) \supset B_{n+1}(y)$ for all $n \in \mathbb{N}$, it follows that for all $i,j>k-1$, $y_{ij}\in V$ and hence $\left\{y_{ij}\right\}$ converges to $y$ in Pringsheim's sense.
Now since $I$ is a strongly admissible ideal of $\mathbb{N} \times \mathbb{N}$, $\left\{y_{ij}\right\}$ $I$-converges to $y$. So for any  open set U containing  $y$, $\left\{(i,j) \in \mathbb{N} \times \mathbb{N}:y_{ij}\notin U\right\} \in I$. Therefore, $\left\{(i,j) \in \mathbb{N} \times \mathbb{N}:y_{ij}\in U\right\} \notin I$, since $I$ is non trivial. Clearly the range set of $\left\{y_{ij}\right\}$ is a subset of the range set of $\left\{x_{ij}\right\}$ and hence $\left\{(i,j) \in \mathbb{N} \times \mathbb{N}:y_{ij}\in U\right\} \subset \left\{(i,j) \in \mathbb{N} \times \mathbb{N}:x_{ij}\in U\right\}$. Therefore $\left\{(i,j) \in \mathbb{N} \times \mathbb{N}:x_{ij}\in U\right\} \notin I$ and hence $(X,\tau)$ is $I$-sequentially compact.\end{proof}

%%%%%%%%%%%%%%%%%%%%%%%%%%%%%%%%%%%%%%%%%%%%%%%%%%%%%%%%%%%%%%%%%%%%%%%%%%%%%%%%%%%%%%%%%%%%%%%%%%%%%%%%%%%%%%%%%%%%%%%%
%%%%%%%%%%%%%%%%%%%%%%%%%%%%%%%%%%%%%%%%%%%%%%%%%%%%%%%%%%%%%%%%%%%%%%%%%%%%%%%%%%%%%%%%%%%%%%%%%%%%%%%%%%%%%%%%%%%%%%%%%

\end{document}